\documentclass{amsart}

\usepackage{amssymb,amsfonts, latexsym,amsmath,hhline,array,longtable}
\usepackage{pdfsync,color, comment,colortbl, mathrsfs,stmaryrd,cite,graphicx}
\usepackage{enumerate}

\usepackage{ifpdf}
\ifpdf \usepackage[colorlinks=true, citecolor=blue, linkcolor=blue, urlcolor=blue]{hyperref} \fi

\newtheorem{formula}{}[section]
\newtheorem{definition}[formula]{Definition}
\newtheorem{corollary}[formula]{Corollary}
\newtheorem{remark}[formula]{Remark}
\newtheorem{lemma}[formula]{Lemma}
\newtheorem{theorem}[formula]{Theorem}

\newtheorem{claim}[formula]{Claim}

\def\thrm{\begin{theorem}}
\def\thrml#1{\begin{theorem}\label{#1}}
\def\ethrm{\end{theorem}}
\def\rmrk{\begin{remark}}
\def\rmrkl#1{\begin{remark}\label{#1}}
\def\ermrk{\end{remark}}
\def\dfntn{\begin{definition}}
\def\dfntnl#1{\begin{definition}\label{#1}}
\def\edfntn{\end{definition}}
\def\nmrt{\begin{enumerate}}
\def\enmrt{\end{enumerate}}

\def\qtnl#1{\begin{equation}\label{#1}}
\def\eqtn{\end{equation}}
\def\lmm{\begin{lemma}}
\def\lmml#1{\begin{lemma}\label{#1}}
\def\elmm{\end{lemma}}
\def\crllr{\begin{corollary}}
\def\crllrl#1{\begin{corollary}\label{#1}}
\def\ecrllr{\end{corollary}}
\def\css{\begin{cases}}
\def\ecss{\end{cases}}
\def\prf{\begin{proof}}
\def\eprf{\end{proof}}

\def\al{\alpha}

\def\la{\lambda}
\def\ge{\geq}
\def\le{\leq}
\def\tt{\theta}

\allowdisplaybreaks

\begin{document}

\title
[A characterization of the Grassmann graphs: one missing case]
{A characterization of the Grassmann graphs:\\ one missing case}

\author{Jack H. Koolen}
\address{University of Science and Technology of China, 
CAS Wu Wen-Tsun Key Laboratory of Mathematics, Hefei, %230026, 
People's Republic of China}
\address{School of Mathematical Sciences, University of Science and Technology of China, Hefei, %230026, 
People's Republic of China}
\email{koolen@ustc.edu.cn}

\author{Chenhui Lv}
\address{School of Mathematical Sciences, University of Science and Technology of China, Hefei, %230026, 
People's Republic of China}
\email{lch1994@mail.ustc.edu.cn}

\author{Alexander L. Gavrilyuk}
\address{University of Memphis, Tennessee, %38152, 
U.S.A.}
\email{a.gavrilyuk@memphis.edu}

\date{}

\begin{abstract}
We prove that the Grassmann graphs $J_2(2D+3,D)$, $D\geq 3$, 
are characterized by 
their intersection numbers, which 
settles one of the few remaining 
cases.
\end{abstract}
%\bigskip
%\noindent
% {\bf Key Words: Grassmann graph, Uniqueness, intersection numbers, cliques, partial linear space}
%\\
%\noindent
% {{\bf 2020 Mathematics Subject Classification: 05E30} }

\maketitle
\newcommand\blfootnote[1]{%
		\begingroup
		\renewcommand\thefootnote{}\footnote{#1}%
		\addtocounter{footnote}{-1}%
		\endgroup}
        

\section{Introduction}\label{s:intro}
Let $\mathbb{F}_q$ be the finite field with $q$ elements, and $V$ be a vector space of dimension $n\geq 2$ over $\mathbb{F}_q$, and let $D$ be an integer satisfying $1\leq D \leq n-1 $. The Grassmann graph $J_q(n, D)$ is the graph with the vertex set consisting of all $D$-dimensional subspaces of $V$,  with two vertices being adjacent if and only if they meet in a subspace of dimension $D-1$. One can easily see that $J_q(n, D)$ is isomorphic to $J_q(n, n-D)$; thus, 
we further assume that $n\geq 2D$, in which case 
$D$ is the diameter of the graph.

The classification of all infinite families of distance-regular graphs 
of unbounded diameter, including the Grassmann graphs, is a major problem in algebraic combinatorics. This problem requires a characterization of the known examples of distance-regular graphs by their intersection numbers. For more results on this and related problems, we refer to the survey paper by Van Dam, Koolen and Tanaka \cite{SurveyDRG}. 

In particular, 
much attention has been paid to this problem
in regard to the Grassmann graphs. 
Note that the Grassmann graphs $J_q(n, 2)$ have the same intersection numbers as the block graphs of $2$-$(\frac{q^n-1}{q-1}, q+1, 1)$ designs, while among the latter ones there are many pairwise non-isomorphic graphs \cite {Jung, Wilson}. Thus, one has to assume $D\geq 3$.
In \cite{km1}, Metsch showed that the Grassmann graph $J_q(n, D), D\geq 3$, is uniquely determined by its intersection numbers unless one of the following cases holds: (1) $n=2D$ or $n=2D+1, q\geq 2$; (2) $n=2D+2, q\in\{2, 3\}$; (3) $n=2D+3, q=2$. In \cite{gk}, Gavrilyuk and Koolen showed that $J_q(2D, D)$ is characterized by its intersection numbers provided that $D$ 
is sufficiently large (see \cite[Theorem~1]{gk} for the precise statement). In \cite{vDK}, Van Dam and Koolen discovered a new family of distance-regular graphs, the twisted Grassmann graphs, which have the same intersection numbers as $J_q(2D+1, D)$ for any prime power $q$.
In this paper, we settle the third exceptional 
case from the Metsch's result, 
i.e., when $n=2D+3$ and $q=2$.

Let $G$ be a distance-regular graph with the same intersection numbers as $J_q(n, D)$. The key idea of Metsch \cite{km1} (see Section \ref{s:pls} 
for more details) was to construct a partial linear space, say $\mathcal{P}$, by taking certain maximal cliques of $G$ as lines; then, by using a result of Ray-Chaudhuri and Sprague \cite{rcs}, to show that the point graph of the partial linear space $\mathcal{P}$ is isomorphic to $J_q(n, D)$. To find large cliques in $G$, which are 
suitable to construct $\mathcal{P}$, Metsch used a counting argument known as the Bose-Laskar type argument (see Theorem \ref{plsm} below).

One can see that the Grassman graph $J_q(n,D)$ 
has two families of maximal cliques. 
When $n=2D$, the cliques from these families 
have the same size, and for $n\in \{2D+1,
2D+2,2D+3\}$, 
they have `almost' the same size. 
This makes Theorem \ref{plsm} is 
no longer applicable, and this is the reason 
why the above idea does not cover the remaining cases. %For $J_q(2D+2, D), q\in\{2, 3\}$, and $J_q(2D+3, D), q=2$, the reason why Metsch's discussion does not work is that there is no integer $s$ satisfying Theorem \ref{plsm}.

In our paper, we will refine Metsch's approach to establish a partial linear space from a graph $G$ that has the same intersection numbers as $J_2(2D+3, D)$. The core idea of our improvement arises from the following observation: if a clique $C$ 
is sufficiently large 
(with the lower bound on its size given 
in terms of the least eigenvalue of $G$), 
then every vertex outside $C$ has either 
too few or too many neighbors from $C$
(see Lemma \ref{cbound}). Based on this, 
we can relax one of the necessary  conditions in the Metsch's Theorem \ref{plsm}, and propose our improved Theorem \ref{pls2}. Using this result, we can establish a partial linear space from $G$. After that, the uniqueness of $J_2(2D+3, D)$ follows by repeating the argument from 
\cite{km1}.

This paper is organized as follows. In Section 2, we recall definitions and basic facts about distance-regular graphs. We also state 
the above mentioned result by Ray-Chaudhuri and Sprague that characterizes the Grassmann graphs as the point graphs of partial linear spaces satisfying certain conditions. In Section 3, we formulate the above mentioned theorem by Metsch, which enables one to construct a partial linear space from a graph. We also give a crucial Lemma \ref{cbound} that, for a graph with a lower bound on the least eigenvalue, controls the number of neighbors of a vertex outside a clique. Using this, we refine the Metsch's result and propose Lemma \ref{pls} and Theorem \ref{pls2}. In Section 4, we apply Theorem \ref{pls2} to prove the uniqueness of $J_2(2D+3, D)$.

\section{Definitions and preliminaries}\label{sect:DP}

\subsection{Basic notation}
All the graphs considered in this paper are finite, undirected and simple. For a given graph $G=(V(G),E(G))$, two vertices are 
\emph{adjacent} if they are joined by an edge and adjacent vertices are also called \emph{neighbors}. For a vertex 
$x\in V(G)$, we denote by $N_G(x)$ the set of 
neighbors of $x$ and put $d_G(x):=|N_G(x)|$, which 
is called the \emph{valency} of $x$ in $G$.
With a slight abuse of notation, 
we also use $N_G(x)$ to denote the \emph{local graph} 
of $G$ at $x$, which is the subgraph induced on $G$ by 
the neighbors of $x$.
A graph $G$ is said to be \emph{regular} with valency $k$ 
if $d_G(x)=k$ holds for all $x\in V(G)$. 

The \emph{adjacency matrix} $A:=A(G)$ of $G$ is the 
$(0,1)$-matrix with rows and columns indexed by $V(G)$, 
such that the $(x,y)$-entry of $A$ is equal to 1 
if and only if $x$ and $y$ are adjacent in $G$. 
The \emph{eigenvalues} of $G$ are those of $A$.
Let $\la_{\min}(G)$ denote the smallest eigenvalue of a graph $G$.  

A \emph{clique} is a complete subgraph, i.e., 
the one induced by
a set of pairwise adjacent vertices. An \emph{anti-clique} 
is a set of mutually nonadjacent vertices. The order 
of a clique or an anti-clique is its cardinality.
Let $m,n$ be positive integers and let $\tilde K_{m,n}$ be 
the graph on $m+n+1$ vertices consisting of a complete graph 
$K_{m+n}$ and a vertex $\infty$ which is adjacent to exactly $m$ verices of the $K_{m+n}$.
Note that $\tilde K_{\ell,n}$ is a subgraph of 
$\tilde K_{m,n}$ for any $\ell\leq m$.

Let $\partial(x,y)=\partial_G(x,y)$ denote 
the path-length distance function for $G$, and 
put $D = \max\{\partial(y, z)\mid y, z \in V(G)\}$, 
called the \emph{diameter} of $G$.

\subsection{Distance-regular graphs}
A connected graph $G$ of diameter $D$ is said to be  
\emph{distance-regular} if there exist integers 
$c_i,a_i,b_i~(0\leq i\leq D)$ such that, for every 
pair of vertices $x,y\in V(G)$ with $\partial(x,y)=i$, 
the vertex $y$ has $c_i$ neighbors at distance $i-1$ from $x$, 
$a_i$ at distance $i$, and $b_i$ at distance $i+1$, respectively. 
The numbers $c_i$, $a_i$ and $b_i$ are called the 
\emph{intersection~numbers} of $G$, and the array 
$\{b_0,b_1,...,b_{D-1}; c_1,c_2,...,c_D\}$ is called 
the \emph{intersection array} of $G$. 
In particular, a distance-regular graph is regular 
with valency $k:=b_0=c_i+a_i+b_i$.

Given an integer $b$, let 
\begin{equation*}
\begin{bmatrix}
  j\\
  1\\
\end{bmatrix}_b= 1+b+b^2+...+b^{j-1}
\end{equation*}
denote the $b$-binomial coefficient. 
With a slight abuse of notation, 
assuming that $b$ is clear from the context, 
we write 
%\begin{equation*}
%\begin{bmatrix}n\\ k\end{bmatrix}_q
$\begin{bmatrix}
j\\
1
\end{bmatrix}$ for the $b$-binomial coefficient.
%:= \begin{bmatrix}
%  j\\
%  1
%  \end{bmatrix}_b$.
%\end{equation*}

A distance-regular graph $G$ of diameter $D$ has \emph{classical parameters}, written 
as a tuple $(D,b,\alpha,\beta)$, if the intersection numbers 
of $G$ can be expressed as follows:
\begin{equation}\label{bi}
 b_i=\left(\begin{bmatrix}
    D \\
    1 
  \end{bmatrix}-
\begin{bmatrix}
    i \\
    1
  \end{bmatrix}\right)
  \left(\beta-\al\begin{bmatrix}
    i \\
    1 
  \end{bmatrix}\right),\quad 0\le i\le D-1,
\end{equation}
\begin{equation}\label{ci}
c_i=\begin{bmatrix}
i\\
1
\end{bmatrix}
\left(1+\al\begin{bmatrix}
i-1\\
1\\
\end{bmatrix}\right),\quad 1\le i\le D.
\end{equation}
By \cite[Corollary 8.4.2]{bcn89}, the distinct eigenvalues of $G$ 
are 
\begin{equation}\label{eigen}
\theta_i=
\begin{bmatrix}
D-i\\
1    
\end{bmatrix}
\left(\beta-\alpha
\begin{bmatrix}
i\\
1    
\end{bmatrix}\right)-
\begin{bmatrix}
i\\
1    
\end{bmatrix}=\frac{b_i}{b^i}-
\begin{bmatrix}
i\\
1    
\end{bmatrix}
,\quad 0\leq i\leq D.\nonumber
\end{equation}
Note that if $b\geq1$, then the eigenvalues $\theta_i$ 
%=\frac{b_i}{b^i}-
%\begin{bmatrix}
%i\\
%1    
%\end{bmatrix}$, 
$(0\leq i\leq D)$ of $G$ are said to be in the \emph{natural ordering}, 
i.e., $k=\theta_0>\theta_1>\cdots>\theta_D$.

The Grassmann graph $J_q(n, D)$, defined in 
Section \ref{s:intro}, 
is distance-regular; all of its intersection numbers can be expressed in terms of $n, D$ and $q$ 
as follows. 

\begin{lemma}[cf.{\cite[Theorem~9.3.3]{bcn89}}]\label{gg}
The Grassmann graph $J_q(n,D),n\geq 2D$, has classical parameters
$$(D,b,\alpha,\beta)=(D,q,q,\begin{bmatrix}
  n-D+1\\
  1\\
  \end{bmatrix}-1).$$
A distance-regular graph with these classical parameters has intersection numbers %given by 
$$b_{j-1}=q^{2j-1}\begin{bmatrix}
  n-D-j+1\\
  1
  \end{bmatrix}\begin{bmatrix}
    D-j+1\\
    1
    \end{bmatrix}, ~c_j=\begin{bmatrix}
      j\\
      1
      \end{bmatrix}^2,\quad (1\leq j\leq D)$$ 
and its eigenvalues and their respective multiplicities are given by 
$$\theta_j=q^{j+1}\begin{bmatrix}
  n-D-j\\
  1\\
  \end{bmatrix}\begin{bmatrix}
    D-j\\
    1\\
    \end{bmatrix}-\begin{bmatrix}
      j\\
      1\\
      \end{bmatrix}, ~m_j=\begin{bmatrix}
        n\\
        j\\
        \end{bmatrix}-\begin{bmatrix}
          n\\
          j-1\\
          \end{bmatrix}\quad (0\leq j\leq D).$$ 
\end{lemma}

%In particular, the Grassmann graph $J_2(2D+3,D)$ has classical parameters $(D,2,2,2^{D+4}-2)$.

The following lemma is a well-known inequality for the smallest 
eigenvalues of local graphs of distance-regular graphs.

\begin{lemma}[cf.{\cite[Theorem~4.4.3]{bcn89}}]\label{localgraph}
Let $G$ be a distance-regular graph of %with valency $k$, 
diameter $D\ge 3$ and with distinct eigenvalues 
$\tt_0>\tt_1>... >\tt_D$.
Then, for a vertex $x$ of $G$, %one has
%let $\la_{\min}(N_G(x))$ be 
%the smallest eigenvalue of the local graph $N_G(x)$ of $G$ at $x$.
%Let $b=\frac{b_1}{\tt_1+1}$. 
\[
\la_{\min}(N_G(x))\ge -\frac{b_1}{\tt_1+1}-1.
\]
If $G$ has classical parameters 
$(D,b,\alpha,\beta)$ with $b\geq 1$, then $\la_{\min}(N_G(x))\geq -b-1$.
%\textcolor{red}{(maybe we need to assume $b>0$ here?---Yes, because we want $\theta_1=\frac{b_1}{b}-1$)}
\end{lemma}

\subsection{Partial linear spaces}\label{s:pls}
An \emph{incidence~structure} $\mathcal{I}=(\mathcal{P}, \mathcal{L}, I)$ consists of a set $\mathcal{P}$ of \emph{points}, a set $\mathcal{L}$ of \emph{lines} (disjoint from $\mathcal{P}$), and a relation $I \subseteq \mathcal{P} \times \mathcal{L}$ called \emph{incidence}. If $(p, L) \in I$, then we say that the point $p$ and the line $L$ are \emph{incident}. 
Two lines are \emph{meeting} if they are incident to a common point.
A \emph{partial linear space} is an incidence structure such that any two points are incident with at most one line. The \emph{point~graph} $G$ of an incidence structure $(\mathcal{P}, \mathcal{L}, I)$ is the graph with the vertex set $\mathcal{P}$ and two distinct points are adjacent if and only if they are incident with a common line. Note that lines induce (maximum) cliques in $G$. For a point $p$ and a line $L$, the \emph{distance} between them is defined as the minimum of the distances $\partial_G(p, p')$, where $p'$ is taken in $L$. The \emph{distance} between two points is the distance between them in the point graph. 

The following theorem, due to Ray-Chaudhuri and Sprague \cite{rcs}, shows a characterization of the Grassmann graphs  as the point graphs of 
partial linear spaces satisfying certain conditions. 
%We present it in the format\footnote{It should be noted that \cite[Theorem~9.3.9]{bcn89} assumes $D\geq \frac{n}{2}$, whereas in this paper, we assume $D\leq\frac{n}{2}$; therefore, the version given here is slightly different, but equivalent.} of Theorem 9.3.9 of \cite{bcn89}. %It should be noted that Theorem 9.3.9 of \cite{bcn89} assumes $D\geq \frac{n}{2}$, whereas in this paper, we assume $D\leq\frac{n}{2}$. %Therefore, the version we provide here is slightly different, i.e., our partial linear space is $(\begin{bmatrix}
%  V\\
%  D
%  \end{bmatrix}, \begin{bmatrix}
%  V\\
%  D-1
%  \end{bmatrix}, \supset)$.

\begin{theorem}[cf.{\cite[Theorem~9.3.9\footnote{Note that \cite[Theorem~9.3.9]{bcn89} assumes $D\geq \frac{n}{2}$, whereas in this paper, we assume $D\leq\frac{n}{2}$; therefore, the statement given here is different, but equivalent.}]{bcn89}}]
\label{rcs}
    Let $(\mathcal{P}, \mathcal{L}, I)$ be a partial linear space such that for some integer $q\geq 2$: 
    \begin{enumerate}[(1)]
        \item each line has at least $q^2+q+1$ points;
        \item each point is on more than $q+1$ lines;
        \item if $p \in \mathcal{P}$ and $L \in \mathcal{L}$ such that $p$ is at distance $1$ from $L$, then there are precisely $q+1$ lines on $p$ meeting $L$;
        \item if $p, p' \in \mathcal{P}$ are at distance $2$, then there are precisely $q+1$ lines on $p$ that are at distance $1$ from $p'$;
        \item the point graph $G$ of $(\mathcal{P}, \mathcal{L}, I)$ is connected.
    \end{enumerate}
    Then $q$ is a prime power and, for some 
    positive integers $n$, $D$, $n\geq 2D$, 
    the graph $G$ is 
    isomorphic to the Grassmann graph $J_q(n,D)$.
    In particular,  
    $(\mathcal{P}, \mathcal{L}, I) \cong (\begin{bmatrix}
  V\\
  D
  \end{bmatrix}, \begin{bmatrix}
  V\\
  D-1
  \end{bmatrix}, \supset)$ for 
  the vector space $V=\mathbb{F}_q^n$, 
  where $\begin{bmatrix}
  V\\
  k
  \end{bmatrix}$ is the set of all $k$-dimensional 
  subspaces of $V$.  
%  \textcolor{red}{
%  We are looking for maximUM cliques. They are given by sets of $d$-spaces, containing given $d-1$ space. So, we should use this partial linear space, 
%  which is different from BCN (note that they use 
%  assumption $e\geq n/2$, but in our case $d\leq n/2$).}
\end{theorem}

Let us recall that $J_q(n, D)$ has two families of maximal cliques: (i) the collections of vertices containing a fixed subspace of dimension $D-1$, and each of them is of size 
$\begin{bmatrix}
  n-D+1\\
  1
  \end{bmatrix}_q$;
%$(q^{n-(D-1)}-1)/(q-1)$; 
(ii) the collections of vertices contained in a fixed subspace of dimension $D+1$, and each of them is of size 
$\begin{bmatrix}
  D+1\\
  1
  \end{bmatrix}_q$.
%$(q^{D+1}-1)/(q-1)$. 
Each edge of $J_q(n, D)$ is contained in a unique clique of each family. 
The partial linear space from the above 
theorem has the maximal cliques of the former 
family as lines, and the vertices as points, 
with the natural incidence relation between them.

%Let $V(\Gamma)$ be the vertex set of $J_q(n, D)$, $\mathcal{L}$  be one of the family of maximal cliques above, then $(V(\Gamma), \mathcal{L}, \in)$ is a partial linear space with the point graph isomorphic to $J_q(n, D)$. The concepts of distance-regular graphs and partial linear spaces will be detailed in Section 2.

Let us again review the strategy used by Metsch \cite{km1} to prove the uniqueness of the Grassmann graphs: first, he applied Theorem \ref{plsm} from Section \ref{sect:BM} to identify sufficiently large cliques in a graph, and defined them as lines, thereby establishing a partial linear space. Second, %through a series of discussions, 
it was verified that this partial linear space satisfies the hypothesis of Theorem \ref{rcs}, 
whence the result follows.%(\textcolor{blue}{This paragraph seems to have some duplication with INTRODUCTION, do you think it still needs to be kept?})

We will adopt a similar strategy. 
However, to achieve a desired result, 
in the next section
we will refine Theorem \ref{plsm} 
in the form of Theorem \ref{pls2}.

\section{An improved Bose-Lasker argument}\label{sect:BM}

The following Theorem \ref{plsm} is given by Metsch \cite{km1}. %, where he presents a method to construct a partial linear space from a graph.

%In the following lemma, for a vertex $x$ not in a clique $C$, we consider the number of neighbors of $x$ in $C$~(see \cite{yk}).
\begin{theorem}[cf.{\cite[Result 2.2]{km1}}]\label{plsm}
Let $s, k,w$ and $e$ be positive integers. Suppose that $G$ is a $k$-regular graph satisfying the following conditions.
\begin{enumerate}[(1)]
  \item Any two adjacent vertices have $w$ common neighbors.
  \item Any two non-adjacent vertices have at most $e+1$ common neighbors.
  \item $w>(2s-1)e-1$.
  \item $k<(s+1)(w+1)-\frac{s(s+1)}{2} e$.
\end{enumerate}
Define a line to be a maximal clique of $G$ 
with at least $w+2-(s-1)e$ vertices. Then every vertex is on at most $s$ lines, and any two adjacent vertices occur together in a unique line.
\end{theorem}

Theorem \ref{plsm} shows that $(V(G), \mathcal{C},\in)$ is a partial linear space, 
where $\mathcal{C}$ is the set of all lines of $G$, 
while $G$ is the point graph of this space.

The following lemma controls the number of neighbors 
of a vertex in a clique of a given order in a graph $G$, provided that 
the smallest eigenvalue of $G$ is bounded from below.

\begin{lemma}[{\rm\hspace{1sp}cf.\cite[Lemma 1.2]{yk}}]\label{cbound}
  Let $\lambda\geq1$ be an integer and $G$ a graph with 
  $\lambda_{\min}(G) \geq - \lambda$. Let $C$ be a clique 
  in $G$ with order $c$. If 
  \[
  c \geq \lambda^4-2\lambda^3+3\lambda^2-3\lambda+3,
  \]
  %$c \geq (\lambda(\lambda-1) )^2 + \lambda(\lambda-1)+(\lambda-1)^2+2=\lambda^4-2\lambda^3+3\lambda^2-3\lambda+3$, 
  then any vertex outside $C$ has either at most $\lambda^2-\lambda$ 
  neighbors in $C$, or at least $c-(\lambda-1)^2$ neighbors in $C$. 
\end{lemma}

\begin{corollary}\label{cbound2}
In the notation of Lemma \ref{cbound},  
  %Let $\lambda \geq 1$ be an integer and $G$ a graph with 
  %$\lambda_{min}(G)\geq -\lambda$. 
  if $\tilde K_{m,n}$ is an 
  induced subgraph of $G$ and 
  $n\geq \lambda^4-2\lambda^3+2\lambda^2-2\lambda+2$, 
  then $m\leq \lambda^2-\lambda$.
\end{corollary}

\begin{proof}
  Assume that $m\geq \lambda^2-\lambda+1$ holds. Then $G$ contains 
  a clique (of $\tilde K_{m,n}$) with at least 
  $m+n \geq \lambda^4-2\lambda^3+3\lambda^2-3\lambda+3$ vertices. 
  Since the vertex $\infty$ of $\tilde K_{m,n}$ is adjacent to $m$ 
  vertices of this clique, it follows by Lemma \ref{cbound} 
  that either $m\geq m+n-(\lambda-1)^2$ or $m\leq \lambda^2-\lambda$. 
  In the former case, one has $n\leq (\lambda-1)^2$, a contradiction, 
  whence the result follows.  
\end{proof}

\begin{corollary}\label{cbound3}
In the notation of Lemma \ref{cbound},  
  %Let $\lambda \geq 1$ be an integer and $G$ a graph with $\lambda_{min}(G)\geq -\lambda$. 
  let $m=\lambda^2-\lambda$ and 
  $n=\lambda^4-2\lambda^3+2\lambda^2-2\lambda+2$.
  Then $G$ is $\tilde K_{m+1,n}$-free.
\end{corollary}

\begin{lemma}\label{pls}
  %Let $s\geq 1$, $m\geq 1$, $n\geq 1$ and $e$ be positive integers. 
  Let $s$, $m$, $n$, $v$, $w$ and $e$ be positive integers. 
  Suppose that $G$ is a graph with $v=|V(G)|$, satisfying the following conditions.
  \begin{enumerate}[(1)]
   \item $G$ is regular with valency $w$.
   \item Any two non-adjacent vertices have at most $e$ common neighbors.
   \item $w>(s-1)e+ms-1$, i.e., $w-(s-1)e+1 \geq ms+1$.
   \item $v<(s+1)(w+1)-\frac{s(s+1)}{2} e$.
   \item $ms+1\geq e+n$.
   \item $e>m$.
   \item $G$ is $ \tilde K_{m+1,n}$-free.
  \end{enumerate}
  Then we have the following assertions.
  \begin{enumerate}[(a)]
    \item Let $s_0$ be the maximum order of an anti-clique in $G$; then $s_0\leq s$.
    \item There are exactly $s_0$ maximal cliques with at least $w-(s-1)m+1$ vertices each, and every vertex of $G$ lies in exactly one of these cliques. 
  \end{enumerate}
\end{lemma}

\begin{proof}
  We prove the lemma with several claims.

\begin{claim}\label{cl:s0}
    $s_0 \leq s$.
\end{claim}
  \begin{proof}
    By {\cite[Lemma 1.1]{km1}}, if there exists an anti-clique of order $s+1$, then $v\geq (s+1)(w+1)-\frac{s(s+1)}{2}e$, contrary to Condition $(4)$.
\end{proof}

%For the sake of brevity 
We shall call a maximal clique \emph{weak} if it has at least $w-(s-1)e+1$ vertices, or \emph{strong} if it has at least $w-(s-1)m+1$ vertices. 
Note that, as $w>(s-1)e+ms-1$, a weak clique has at least $ms+1$ vertices.

\begin{claim}\label{cl:clique-inter}
    If $C$ is a weak clique and $x\notin C$, then $|N_G(x)\cap V(C)|\leq m$. 
    In particular, $|V(C)\cap V(C')| \leq m$ for any other weak clique $C'$.
    %If $C$ and $C'$ are two weak cliques, $x$ is a vertex does not lie in $C$, then $|N_G(x)\cap V(C)|\leq m$ and $|V(C)\cap V(C')| \leq m$.
\end{claim}
\begin{proof}
  Clearly, a weak clique has at least $ms+1\geq e+n\geq m+n+1$ vertices. 
  By Condition (2) and the maximality of $C$, it follows that 
  $|N_G(x)\cap V(C)|\leq e$. By Condition (7), 
  we have 
  $|N_G(x)\cap V(C)|\leq m$ and 
  the claim follows.
\end{proof}

\begin{claim}\label{cl:s0cliques}
    Each vertex lies in a weak clique and there are exactly $s_0$ weak cliques.
\end{claim}
\begin{proof}
  We prove Claim \ref{cl:s0cliques} in several steps.
\begin{enumerate}[Step 1]
  \item Suppose that $I=\{x_1,x_2,...,x_{s_0}\}$ is an anti-clique 
  of maximum order $s_0$. Then the set 
  \[
  P_i:=\{y \in V(G)\mid~y\text{ is nonadjacent to every vertex of }I\backslash\{x_i\}\}
  \]
   is a clique with at least $w-(s-1)e+1$ vertices for $i=1,2,...,s_0$. 
   Furthermore, $x_i \in P_i$, and $P_i$ is contained in a weak clique, 
   say $C_i$, for $i=1,2,...,s_0$.
  \begin{proof}
    For any $i\in\{1,2,...,s_0\}$, if there exist two nonadjacent 
    vertices $y_1$ and $y_2$ of $P_i$, then the set $(I-\{x_i\})\cup \{y_1,y_2\}$ 
    is an anti-clique of order $s_0+1$, a contradiction with Claim \ref{cl:s0}. 
    Therefore, $P_i$ is a clique containing $x_i$ and all vertices of $N_G(x_i)$ except those which are adjacent to some vertex of $I\backslash\{x_i\}$. 
    By Condition (2), it follows that $P_i$ has at least 
    $|N_G(x_i)|+1-(s_0-1)e\geq w+1-(s_0-1)e\geq w+1-(s-1)e$ vertices, and we may assume that 
    the set 
    $P_i$ is contained in some weak clique $C_i$.
  \end{proof} 
  
  \item $C_i$ is the unique weak clique containing $P_i$ for $i=1,2,...,s_0$.
  \begin{proof}
    Assume $C_i$ and $C'$ are different weak cliques with $V(P_i)\subseteq V(C_i)\cap V(C')$. By Claim \ref{cl:clique-inter} and Conditions (3), 
    (5) and (6), one has  
    $m\geq |V(C_i)\cap V(C')|\geq |V(P_i)|\geq w-(s-1)e+1\geq ms+1$, %\geq e+n>m+n$, 
    a contradiction.
  \end{proof}
  
  \item $V(G)=V(C_1)\cup V(C_2)...\cup V(C_{s_0})$.
  \begin{proof}
    Assume there exists a vertex $x_0$ such that $x_0\notin V(C_1)\cup V(C_2)...\cup V(C_{s_0})$. Let $I_0$ be a largest anti-clique 
    %of maximum order 
    %in $\{J|J$ is a maximal anti-clique with $x_0 \in J\}$. 
    containing $x_0$. Obviously $|I_0|\leq s_0$.

    Note that each of $C_i$ contains one vertex of $I_0$, $i=1,2,...,s_0$. 
    Indeed, if there exists $C_i$ with $V(C_i)\cap I_0=\varnothing $, then 
    by Claim \ref{cl:clique-inter}, any vertex in $I_0$ has at most $m$ neighbors 
    in $C_i$. As $|V(C_i)|\geq w-(s-1)e+1 \geq ms+1 \geq m|I_0|+1$, there exists a vertex $z\in C_i$ such that $z$ is nonadjacent to any vertex of $I_0$, and hence we obtain a larger anti-clique $\{z\}\cup I_0$, a contradiction.

    As each of $C_i$, $i=1,2,...,s_0$, 
    contains one vertex of $I_0$ except $x_0$ 
    and $|I_0|\leq s_0$, it follows that 
    %As $|I_0\backslash \{x_0\}| \leq s_0-1$, 
    there exists a vertex $u_0\in I_0\backslash\{x_0\}$ such that $u_0$ lies in at least two of $C_1,C_2,...,C_{s_0}$. Without loss of generality, we assume $u_0 \in V(C_1)\cap V(C_2)$.

    By Claim \ref{cl:clique-inter}, any vertex of $I_0\backslash \{u_0\}$ has 
    at most $m$ neighbors in $C_i$, $i=1,2$, and $|V(C_1)\cap V(C_2)|\leq m$. 
    As $|V(C_1)|-m-(|I_0|-1)m\geq ms+1-s_0m\geq 1$, there exists $u_1\in V(C_1)\backslash V(C_2)$ such that $u_1$ is not adjacent to any vertex of $I_0\backslash \{u_0\}$. Similarly, as $V(C_1)\cap V(C_2)\subseteq N_G(u_1)\cap V(C_2)$ and $|V(C_2)|-m-(|I_0|-1)m\geq ms+1-s_0m\geq 1$, there exists $u_2\in V(C_2)\backslash V(C_1)$ such that $u_2$ is not adjacent to any vertex of $(I_0-\{u_0\})\cup \{u_1\}$. Now $(I_0-\{u_0\})\cup \{u_1,u_2\}$ is a larger 
    anti-clique containing $x_0$, a contradiction with the choice of $I_0$.
  \end{proof}
  
  \item $C_1,C_2,...,C_{s_0}$ are all the weak cliques in $G$.
  \begin{proof}
    Suppose there exists a weak clique $C'$, distinct from $C_1,C_2,...,C_{s_0}$. 
    By Claim \ref{cl:clique-inter}, 
    $|V(C')\cap V(C_i)|\leq m$ for $i=1,2,...,s_0$. 
    By Step 3, %Claim \ref{cl:s0cliques}, 
    $|V(C')|\leq \sum_{i = 1}^{s_0} |V(C')\cap V(C_i)|\leq ms_0\leq ms $, 
    a contradiction with the definition of a weak clique.
  \end{proof}

\end{enumerate}
Following the above steps, Claim \ref{cl:s0cliques} is proved.
\end{proof}

\begin{claim}
    The weak cliques $C_1,C_2,...,C_{s_0}$ are strong.
    %$|V(C_i)|\geq w-(s-1)m+1$ for $i=1,2,...,s_0$, i.e., 
    %weak cliques are strong cliques.
\end{claim}
\begin{proof}
%\textcolor{red}{here we need a vertex which does not belong to $C_j$. 
%Hence I suggest to change $u$ to $x_i$. No problem. I think it's more appropriate.}
  By Claim \ref{cl:clique-inter}, the vertex $x_i\in C_i$ has at most $m$ 
  neighbors in $C_j$ for $j\neq i$. Hence, by Claim \ref{cl:s0cliques}, 
  $|V(C_i)|\geq d_G(x_i)+1-(s_0-1)m \geq w+1-(s_0-1)m\geq w-(s-1)m+1$; 
  thus, by definition, $C_i$ is a strong clique. 
%  By Claim \ref{cl:clique-inter}, any vertex $u\in C_i$ has at most $m$ 
%  neighbors in $C_j$ for $j\neq i$. Hence, by Claim \ref{cl:s0cliques}, 
%  $|V(C_i)|\geq d_G(u)+1-(s_0-1)m \geq w+1-(s_0-1)m\geq w-(s-1)m+1$.
\end{proof}

\begin{claim}
    $V(C_i)\cap V(C_j)=\varnothing$ for $i\neq j$.
\end{claim}
\begin{proof}
  Assume there exists a vertex $u\in V(C_i)\cap V(C_j)$ for $i\neq j$. 
  By Claim \ref{cl:clique-inter}, $|V(C_i)\cap V(C_j)|\leq m$. Since 
  $V(C_i)\cup V(C_j)\subseteq N_G(u)\cup \{u\}$, we obtain 
  $|V(C_i)\cup V(C_j)|\leq w+1$. Therefore, 
  \[
  w+m+1\geq |V(C_i)\cap V(C_j)|+|V(C_i)\cup V(C_j)|=|V(C_i)|+|V(C_j)|\geq 2w-2(s-1)m+2,
  \]
  i.e., $0\geq w-(2s-1)m+1$. However, by Conditions (3) and (6), $w>(s-1)e+ms-1\geq (s-1)m+ms-1\geq (2s-1)m-1$, a contradiction.
\end{proof}

Following the above discussion, Lemma \ref{pls} is proved.
\end{proof}

The next result follows directly from Lemma \ref{pls} and represents 
an improvement over Theorem \ref{plsm} in the following manner: 
we replace the condition $w>(2s-1)e-1$ of Theorem \ref{plsm} with 
$w>(s-1)e+ms-1$. Moreover, since we also require $e>m$, 
the range of $s$ satisfying the conditions is extended. 
The key factor contributing to this improvement is our observation 
that the local graph $N_G(x)$ is $\tilde K_{m+1,n}$-free. We will then use Theorem \ref{pls2} to establish a partial linear space.

\begin{theorem}\label{pls2}
%Let $s \geq 1$, $m\geq 1$ and $n\geq 1$ and let $k,w$ and $e$ be non-negative integers. 
Let $s$, $m$, $n$, $k$, $w$ and $e$ be positive integers. 
Suppose that $G$ is a $k$-regular graph satisfying the following conditions.
\begin{enumerate}[(1)]
  \item Any two adjacent vertices have $w$ common neighbors.
  \item Any two non-adjacent vertices have at most $e+1$ common neighbors.
  \item $w>(s-1)e+ms-1$.
  \item $k<(s+1)(w+1)-\frac{s(s+1)}{2} e$.
  \item $ms+1\geq e+n$.
  \item $e>m$.
  \item For every vertex $x\in V(G)$, its local graph $N_G(x)$ is $\tilde K_{m+1,n}$-free. 
\end{enumerate}
%Define a line to be a maximal clique $C$ in $G$ satisfying $|V(C)|\geq w+2-(s-1)m$. 
Define a line to be a maximal clique of $G$ 
with at least $w+2-(s-1)m$ vertices.
Then every vertex is on at most $s$ lines, and any two adjacent vertices occur together in a unique line.
\end{theorem}

\section{Uniqueness of $J_2(2D+3,D)$} \label{sect:MR}

In this section, we prove our main result.

\begin{theorem}\label{main}
Let $G$ be a distance-regular graph with classical parameters 
$(D,b,\alpha,\beta)$ where $D\geq 3$, $b=2$, $\alpha=2$, and 
$\beta \geq \begin{bmatrix}
    D+4\\
    1
\end{bmatrix}_2-1$. %=2^{D+4}-2$. 
Then 
$\beta=\begin{bmatrix}
    D+\ell+1\\
    1
\end{bmatrix}_2-1$
%$\beta=2^{D+\ell+1}-2$ 
for some integer $\ell\geq 3$, and $G$ is the Grassmann graph $J_2(2D+\ell, D)$. 
In particular, $J_2(2D+3, D)$ is uniquely determined by its intersection numbers.  
\end{theorem}
\begin{proof}
  %By Lemmma \ref{gg}, 
  By Eqs. \eqref{bi} and \eqref{ci},
  $G$ has intersection numbers $b_0=(2^D-1)\beta$, $b_1=(2^D-2)(\beta-2)$, $a_1=b_0-c_1-b_1=\beta-1+2(2^D-2)$ and $c_2=9$.
  
  Put $s:=\frac{5}{4}\cdot 2^D$, $m:=6$, $n:=41$, $k:=b_0=(2^D-1)\beta$, $w:=a_1=\beta-1+2(2^D-2)$ and $e:=c_2-1=8$. We shall show that all the conditions of Theorem \ref{pls2} are satisfied.

  Since $D\geq 3$, we have $w-(s-1)e-ms+1=w-(14s-8)+1=\beta-1+2(2^D-2)-\frac{35}{2}\cdot2^D+9>0$, as $\beta\geq 2^{D+4}-2$. Hence Condition (3) 
  of Theorem \ref{pls2} is satisfied.

  Since $D \geq 3$, we have:
\begin{align*}
&(s+1)(w+1) - \frac{s(s+1)e}{2} - k \\
&= (s+1)(w+1-4s) - k \\
&= \left( \frac{5}{4}\cdot2^D + 1 \right) \left( \beta + 2(2^D - 2) - 5 \cdot 2^D \right) - (2^D - 1) \beta \\
%&= \left( \frac{5}{4}\cdot2^D + 1 \right) \beta - \left( \frac{5}{4}\cdot2^D + 1 \right) (3 \cdot 2^D + 4) - (2^D - 1) \beta \\
&= \left( \frac{1}{4}\cdot2^D + 2 \right) \beta - \left( \frac{5}{4}\cdot2^D + 1 \right) (3 \cdot 2^D + 4) \\
&\geq \left( \frac{1}{4}\cdot2^D + 2 \right) \left( 2^{D+4} - 2 \right) - \left( \frac{5}{4}\cdot2^D + 1 \right) (3 \cdot 2^D + 4) \\
&= 4 \cdot 2^{2D} + (32 - \frac{1}{2}) \cdot 2^D - 4 - \left( \frac{15}{4}\cdot2^{2D} + 8 \cdot 2^D + 4 \right) \\
&= \frac{1}{4}\cdot2^{2D} + (24 - \frac{1}{2}) \cdot 2^D - 8 > 0.
\end{align*}
Hence Condition (4) of Theorem \ref{pls2} is satisfied.

  Since $D\geq 3$, we have $ms+1-e-n=6\cdot\frac{5}{4}\cdot2^D+1-8-41=6\cdot(\frac{5}{4}\cdot2^D-8)>0$. Hence Condition (5) of Theorem \ref{pls2} is satisfied.

  For a vertex $x\in V(G)$, it follows by Lemma \ref{localgraph} that 
  $\lambda_{min}(N_G(x))\geq-3$. Therefore, by Corollary \ref{cbound3}, 
  the local graph $N_G(x)$ is $\tilde K_{7,41}$-free, 
  and Condition (7) of Theorem~\ref{pls2} is satisfied.
  
  Obviously, Conditions (1), (2) and (6) of Theorem~\ref{pls2} are satisfied. By Theorem~\ref{pls2}, define a line to be a maximal clique $C$ satisfying $|V(C)|\geq w+2-(s-1)m$. Then every vertex of $G$ lies on at most $s$ lines, and any two adjacent vertices lie on a unique line, i.e., we have constructed a partial linear space.

  Now the proof can be completed following 
  the argument of Metsch \cite{km1}.
\end{proof}

%\textcolor{red}{Maybe we should briefly 
%comment on $2D+2$ case, why it is not covered.}
\begin{remark}
    The method we present here does not apply to $J_2(2D+2, D)$ and $J_3(2D+2, D)$, the other remaining cases. For $J_2(2D+2, D)$, there is no integer $s$ that satisfies Condition (4) of Theorem \ref{pls2}. For $J_3(2D+2, D)$, there is no integer $s$ that satisfies both Condition (3) and Condition (4) of Theorem \ref{pls2} at the same time.
\end{remark}

\section*{Acknowledgements}  
\noindent J.H. Koolen is partially supported by the National Key R. and D. 
Program of China (No. 2020YFA0713100), the National Natural Science Foundation 
of China (No. 12471335), and the Anhui Initiative in Quantum Information Technologies (No. AHY150000).

\clearpage

\end{document}